\documentclass[11pt,a4paper]{amsart}

\usepackage[utf8]{inputenc}
\usepackage[english]{babel}
\usepackage{amssymb, amsmath, amsthm, amsfonts, bm, bbm, dsfont}
\usepackage{bm}
\usepackage{tikz}
\usepackage{tikz-cd}
\usepackage[noadjust]{cite}
\usepackage{color}
\usepackage[]{hyperref}
\hypersetup{
	colorlinks=true,
	linkcolor=blue,
	citecolor=green,
	urlcolor=magenta
}

\newcommand{\nocontentsline}[3]{}
\newcommand{\tocless}[2]{\bgroup\let\addcontentsline=\nocontentsline#1{#2}\egroup}
\newcommand{\ab}{{\mathbf{Ab}}}
\newcommand{\vects}[1]{{\mathbf{Vect}_{#1}}}
\newcommand{\cgv}{{\CG}}
\newcommand{\sav}{{{\CS \CA}_\BQ}}
\newcommand{\omot}{\mathbf{M_1}}
\newcommand{\abgen}[1]{{\langle #1 \rangle}}
\newcommand{\rk}[3]{{\mathrm{rk}_{#1}{\left(#2,#3\right)}}}
\newcommand{\dimq}[1]{{\dim_\BQ{\left(#1\right)}}}
\newcommand{\dimqb}[1]{{\dim_\qb{\left(#1\right)}}}
\newcommand{\la}{\langle}
\newcommand{\ra}{\rangle}
\newcommand{\Hom}[3]{{\mathrm{Hom}_{#1}{\left(#2,#3\right)}}}
\newcommand{\End}[2]{{\mathrm{End}_{#1}{\left(#2\right)}}}
\newcommand{\Ext}[3]{{\mathrm{Ext}^1_{#1}{\left(#2,#3\right)}}}
\newcommand{\Extn}[4]{{\mathrm{Ext}^{#2}_{#1}{\left(#3,#4\right)}}}
\renewcommand{\ker}[1]{{\mathrm{ker}{\left(#1\right)}}}
\newcommand{\img}[1]{{\mathrm{im}{\left(#1\right)}}}
\newcommand{\qb}{\ov{\BQ}}
\newcommand{\surj}{{\ \twoheadrightarrow \ }}
\newcommand{\inj}{{\ \hookrightarrow \ }}
\newcommand{\isom}{{\ \xrightarrow{\sim} \ }}
\newcommand{\hand}{{ \quad \mathrm{and} \quad }}
\newcommand{\seq}{{\ \subseteq \ }}

\newcommand{\ov}[1]{{\overline{#1}}}

\newcommand{\ti}[1]{{\tilde{#1}}}
\newcommand{\BC}{{\mathbb{C}}}
\newcommand{\BG}{{\mathbb{G}}}
\newcommand{\BN}{{\mathbb{N}}}
\newcommand{\BQ}{{\mathbb{Q}}}

\newcommand{\BZ}{{\mathbb{Z}}}
\newcommand{\CA}{{\mathcal{A}}}
\newcommand{\CB}{{\mathcal{B}}}
\newcommand{\CC}{{\mathcal{C}}}
\newcommand{\CG}{{\mathcal{G}}}
\newcommand{\CP}{{\mathcal{P}}}
\newcommand{\CS}{{\mathcal{S}}}
\newcommand{\CT}{{\mathcal{T}}}

\newtheorem{thm}{Theorem}[section]
\newtheorem{cor}[thm]{Corollary}
\newtheorem{lem}[thm]{Lemma}
\newtheorem{prop}[thm]{Proposition}
\newtheorem*{prop*}{Proposition}
\theoremstyle{definition}
\newtheorem{defn}[thm]{Definition}
\newtheorem{rem}[thm]{Remark}
\newtheorem{ex}[thm]{Example}
\newtheorem{para}[thm]{}
\newtheorem{conv}[thm]{Convention}

\title{Ranks of $1$-motives as dimensions of Ext$^1$ vector spaces}
\author{Nicola Nesa}
\email{nicola.nesa@gmail.com}
\date{\today}

\begin{document}

\maketitle
\thispagestyle{empty}

\tocless{\section*}{Abstract}

We reformulate the ``ranks'' that appear in the dimension formula for the linear space of periods of a $1$-motive established in \cite[Theorem 1.4]{HW} in a more conceptual and categorical way, as dimensions of $\mathrm{Ext}^1$ vector spaces. This constitutes the first step towards rewriting the dimension formula purely in general categorical terms, rather than through definitions and computations introduced \textit{ad hoc} for $1$-motives.\\

\tableofcontents

\newpage
\pagestyle{headings}
\setcounter{page}{1}
\section*{Introduction}

\noindent \textbf{Results}

In this article we prove the results that follow this paragraph. The definitions of the various ranks of $1$-motives (``rk'' in the formulas below) are provided in Definition \ref{defn:Rank}, and, more generally, the notation for $1$-motives is recalled in Section \ref{s:recap}. It is worth pointing out that $\abgen{M}$ denotes the category containing the $1$-motives that are subquotients of finite direct sums of copies of $M$ (see Convention \ref{conv:Cats}) and that the Ext$^1$ vector spaces in the formulas below are computed in that category.

\begin{prop*}[{\ref{prop:RankExtOmotBZ}}]
	Let $M=(X,A,T,G,u)$ be a $1$-motive with $A = 0$, $T \neq 0$ and $X \neq 0$. Then
	\[ \rk{\BG_m}{X}{M} = \dimq {\Ext{\la M \ra}{[\BZ]}{[\BG_m]}}.\]
\end{prop*}

\begin{prop*}[{\ref{prop:RankExtOmot}}]
	Let $M=(X,A,T,G,u)$ be a $1$-motive with $A \neq 0$ and $T \neq 0$, and let $B$ be a simple component of $A$. Then
	\[ \rk{B}{T}{M} = \dim_{\End{}{B}} {\left(\Ext{\la G \ra}{B}{\BG_m}\right)} = \dim_{\End{}{B}} \left(\Ext{\la M \ra}{[B]}{[\BG_m]} \right).\]
\end{prop*}

\begin{prop*}[{\ref{prop:RankExtOmotX}}]
	Let $M=(X,A,T,G,u)$ be a $1$-motive with $A \neq 0$ and $X \neq 0$, and let $B$ be a simple component of $A$. Then
	\[ \rk{B}{X}{M} = \dim_{\End{}{B}} {\left(\Ext{\la M \ra}{[\BZ]}{[B]}\right)}.\]
\end{prop*}

These propositions all rely on the following general result, which we prove in Section \ref{s:abelian} (be mindful that $A$ and $B$ in the following proposition are not related to $A$ and $B$ in the previous ones):

\begin{prop*}[{\ref{prop:ExtGeneral}}]
	Let $\CB$ be an abelian category with full, semisimple abelian subcategories $\CA$ and $\CC$, such that $\Hom{\CB}{A}{C} = 0$ for all objects $A$ in $\CA$ and $C$ in $\CC$. Let $B$, $A$, $A'$, $C$, $C'$ be objects in $\CB$, $\CA$, $\CC$, respectively, such that $B \in \Ext{\CB}{C}{A}$ and that $A'$ and $C'$ are also contained in $\la B \ra$. Then
	\[ \Ext{\la B \ra}{C'}{A'} = \left\{ \sum_{i=1}^n \phi_i B\psi_i \ \middle\vert \ n \in \BN, \phi_i \in \Hom{\CB}{A}{A'}, \psi_i \in \Hom{\CB}{C'}{C} \right\} \]
	as subgroup of $\Ext{\CB}{C'}{A'}$.
\end{prop*}

\noindent \textbf{Historical context and motivation}

The study of periods lies at the intersection between algebraic geometry and transcendental number theory: the complex numbers called \textit{periods} arise as integrals of algebraic differential forms on algebraic varieties and form a countable $\qb$-subalgebra of $\BC$. Examples of periods include all the algebraic numbers and their logarithms, $\pi$, all values of the Riemann zeta function at integers, and in fact most numbers that are of interest in mathematics and mathematical physics.

For every algebraic variety $X$ defined over $\qb$, the $\qb$-algebra generated by the periods of $X$ has finite transcendence degree. The first mention of the period conjecture, whose aim is to determine this transcendence degree, dates back to \cite[p.~101]{GrothPerConj} (see \cite{AndLetter} for a detailed historical overview). Because of their well-behaved formal properties, the right categories to study periods are categories of motives, rather than categories of algebraic varieties. The simplest case is well-understood: for Artin motives, one obtains the algebraic numbers as periods. The next step is to study $1$-motives, defined in \cite[§10]{Del3}, but already in this case the period conjecture in its original formulation is currently intractable. However, there is also a variant of this conjecture, which we call the \textit{linear} period conjecture, which was introduced by Kontsevich in \cite[§4]{KZ} and turns out to be closely related to the original, or \textit{algebraic}, period conjecture due to Grothendieck (see the discussion in \cite{AndLetter} and \cite[§5]{HubGalois}). The conjecture of Kontsevich can be elegantly phrased as saying that the only linear relations between periods are the ones implied by the fact that integration is bilinear and functorial.

In \cite{HW}, Huber and W\"ustholz succeed in proving the linear period conjecture for $1$-motives and derive an explicit formula for the dimension of the $\qb$-linear space generated by the periods of a $1$-motive. These results are very valuable in their own right and have several important applications in classical transcendence theory, but new techniques would be needed to generalize them to the algebraic period conjecture or to other categories of motives. In particular, the explicit formula of \cite[Theorem 1.4]{HW} for the dimension of the $\qb$-linear space generated by the periods of a $1$-motive is quite complicated and relies heavily on the explicit definition of $1$-motives.

The goal of this article is to reformulate some terms in the formula of \cite[Theorem 1.4]{HW} (reproduced here as Theorem \ref{thm:Formula}), called ``ranks'' and defined \textit{ad hoc} for $1$-motives, in more abstract and categorical terms. We do it in Propositions \ref{prop:RankExtOmotBZ}, \ref{prop:RankExtOmot} and \ref{prop:RankExtOmotX}  (see above), by showing that these ranks are the dimensions of some Ext$^1$ vector spaces (note that, even if the objects appearing as arguments of the Ext$^1$ bifunctor come from the explicit definition of $1$-motives, they can be expressed in general and categorical terms using the weight filtration). The main ingredient is Proposition \ref{prop:ExtGeneral} (see above), which holds in any abelian category. The results of this article are part of my PhD thesis, more precisely \cite[§3]{NesaPhD}.

We have not achieved completely the goal of understanding the whole dimension formula from a more abstract and conceptual point of view, since there is one summand in the formula which cannot be directly reformulated using our results; however, Huber and Kalck have recently addressed this problem in \cite{HK}.

Besides making the dimension formula for $1$-motives much clearer, this approach consisting of abstraction and categorization might, in the future, find application to more complicated categories of motives.\\

\noindent \textbf{Acknowledgments}

The results of this article were obtained while I was a PhD student at the University of Freiburg. The funding was provided by the DFG-Graduiertenkolleg 1821 ``Cohomological Methods in Geometry''.

Thanks to my PhD supervisor, Annette Huber, for providing detailed feedback on this article even after the completion of my PhD.

Thanks to the anonymous reviewer for their useful comments.

\newpage
\section{Results for Ext$^1$ groups in abelian categories}\phantom{}
\label{s:abelian}

The goal of this section is to formulate and prove Proposition \ref{prop:ExtGeneral}, which holds in any abelian category and constitutes the main ingredient for the proofs of Propositions \ref{prop:RankExtOmotBZ}, \ref{prop:RankExtOmot} and \ref{prop:RankExtOmotX}. In this section we use only basic homological algebra.

\begin{conv} \label{conv:Cats}
	In this article, categories are always essentially small. We are interested in abelian categories, such as the category $\omot$ of $1$-motives (see Definition \ref{defn:Omot}). If $\CB$ is an abelian category and $B$ is an object of $\CB$, we denote by $\abgen{B}$ the smallest full abelian subcategory of $\CB$ closed under subquotients containing $B$; it consists of the objects of $\CB$ that are subquotients of finite direct sums of copies of $B$.
\end{conv}

\begin{rem}
	The categories in the later sections are actually also $\BQ$-linear: in $\BQ$-linear abelian categories, Hom groups and Ext groups are even $\BQ$-vector spaces. Abelian subcategories of $\BQ$-linear abelian categories are automatically $\BQ$-linear.
\end{rem}

\begin{para} \label{para:Ext}
	Let $\CB$ be a abelian category and let $A$, $A'$, $C$, $C'$ be objects of $\CB$. We denote by $\Ext{\CB}{C}{A}$ the abelian group of \textit{isomorphism classes of (Yoneda) extensions} of $C$ by $A$, where addition is given by the \textit{Baer sum}. Given $B \in \Ext{\CB}{C}{A}$, $\phi \in \Hom{\CB}{A}{A'}$ and $\psi\in \Hom{\CB}{C'}{C}$, we denote by
	\[\phi B:= A' \sqcup_A B \in \Ext{\CB}{C}{A'} \hand B\psi:= B \times_C C' \in \Ext{\CB}{C'}{A}\]
	the push-out of $B$ along $\phi:A \to A'$ and the pull-back of $B$ along $\psi:C' \to C$, respectively.
\end{para}

\begin{lem} \label{lem:Retraction}
	Consider the following commutative diagram with short exact rows in an abelian category $\CB$:
	\[\begin{tikzcd}
		A' \arrow[d, "\iota"'] \arrow[r] & B' \arrow[d] \arrow[r] & C' \arrow[d, "\psi"] \\
		A \arrow[r]                  & B \arrow[r]                  & C.                
	\end{tikzcd}\]
	If $\iota$ has a retraction $\kappa$, then $B' = \kappa B \psi$ as elements of $\Ext{\CB}{C'}{A'}$. Dually, if $\psi$ has a section $\phi$, then $B = \iota B' \phi$ as elements of $\Ext{\CB}{C}{A}$.
\end{lem}

\begin{proof}
	A standard argument implies that $B\psi=\iota B'$ as elements of $\Ext{\CB}{C'}{A}$. Then it suffices to apply to both sides of the latter equality either the push-out along $\kappa$ or the pull-back along $\phi$ to obtain the desired equalities of isomorphism classes of extensions.
\end{proof}

Now we state and prove the main result of this section. Pay attention to the lower index (the category) in the $\Ext{-}{-}{-}$ groups, as it is crucial for the argument.

\begin{prop} \label{prop:ExtGeneral}
	Let $\CB$ be an abelian category with full, semisimple abelian subcategories $\CA$ and $\CC$, such that $\Hom{\CB}{A}{C} = 0$ for all objects $A$ in $\CA$ and $C$ in $\CC$. Let $B$, $A$, $A'$, $C$, $C'$ be objects in $\CB$, $\CA$, $\CC$, respectively, such that $B \in \Ext{\CB}{C}{A}$ and that $A'$ and $C'$ are also contained in $\la B \ra$. Then
	\[ \Ext{\la B \ra}{C'}{A'} = \left\{ \sum_{i=1}^n \phi_i B\psi_i \ \middle\vert \ n \in \BN, \phi_i \in \Hom{\CB}{A}{A'}, \psi_i \in \Hom{\CB}{C'}{C} \right\} \]
	as subgroup of $\Ext{\CB}{C'}{A'}$.
\end{prop}

\begin{proof}
	The subcategory $\la B \ra$ of $\CB$ is closed under pull-backs, push-outs, and sums of extensions, so every extension of $C'$ by $A'$ of the form $\sum_{i=1}^n \phi_iB\psi_i$ is indeed contained in $\Ext{\la B \ra}{C'}{A'}$. To prove the proposition we need to show that the converse inclusion also holds. To this aim, let $B'$ be an object of $\abgen{B}$ that extends $C'$ by $A'$. By definition of the category $\abgen{B}$, there is some object $B''$ in $\CB$ and some $n \in \BN$ such that
	\[B' \twoheadleftarrow B'' \inj B^n\]
	is a subquotient of a power of $B$. Setting
	\[A'':= A^n \times_{B^n} B'' = \ker{A^n \to B^n/B''} \in \CA, \ \mathrm{and}\]
	\[C'':= B''/A'' = \img{B'' \to C^n} \in \CC,\]
	the subobject $B'' \inj B^n$ determines a commutative diagram with short exact rows of the form
	\[\begin{tikzcd}
		A'' \arrow[r]  \arrow[d, hook, "f"]     & B'' \arrow[r]   \arrow[d, hook]       & C'' \arrow[d, hook, "g"] \\
		A^n \arrow[r]       & B^n \arrow[r]          & C^n.               
	\end{tikzcd}\]
	Since $\CA$ is semisimple and $f$ is injective, $f$ has a retraction $r$. It thus follows from Lemma \ref{lem:Retraction} that $B'' = r (B^n) g$ as elements of $\Ext{\la B \ra}{C''}{A''}$. Now we want to apply a dual argument to the quotient map $B'' \surj B'$, but some care is needed: a priori, we only know that there is \textit{some} commutative diagram with short exact rows of the form
	\[\begin{tikzcd}
		\ti{A} \arrow[r] &B' \arrow[r] &\ti{C}\\
		A'' \arrow[r]  \arrow[u, two heads] & B'' \arrow[r] \arrow[u, two heads] & C'', \arrow[u, two heads]
	\end{tikzcd}\]
	with some objects $\ti{A} \in \CA$ and $\ti{C} \in \CC$, but we do not know anything about the upper short exact sequence. However, since we have assumed that there is no non-trivial morphism from $\CA$ to $\CC$, a standard argument implies that any given object of $\CB$ can define at most one extension of an object of $\CC$ by an object of $\CA$, up to (unique) isomorphisms; therefore, we can assume without loss of generality that the upper short exact sequence is the same that defines $B' \in \Ext{\CB}{C'}{A'}$, with $\ti{A} = A'$ and $\ti{C}=C'$. Now, from Lemma \ref{lem:Retraction} we can deduce, as above (this time, using that $\CC$ is semisimple), that $B' = \phi (B^n)\psi$ for some maps $\phi \in \Hom{\CB}{A^n}{A'}$, $\psi \in \Hom{\CB}{C'}{C^n}$. To conclude, for $i \in \{1,...,n\}$, denote by $\iota_i : A \to A^n$ and $\rho_i: C^n \to C$ the standard inclusions and projections. Set $\phi_i := \phi \circ \iota_i \in \Hom{\CB}{A}{A'}$ and $\psi_i := \rho_i \circ \psi \in \Hom{\CB}{C'}{C}$. Then, using that extensions are compatible with direct sums, we compute that
	\[ B'= \phi (B^n)\psi = \sum_{i =1}^n \phi_i B \psi_i\]
	as an element of $\Ext{\la B \ra}{C}{A}$, as was to be shown.
\end{proof}

\section{Recapitulation on $1$-motives and their linear spaces of periods}
\label{s:recap}

In this section we recapitulate the main definitions and notation that we need to explain and prove our results in Section \ref{s:results}.

\begin{conv} \label{conv:CatsSAV}
	The base field throughout this article is $\qb$. Recall that a \textit{semiabelian variety} $G$ is a commutative, connected group variety which fits into a short exact sequence of group varieties
	\[ 0 \to T \to G \to A \to 0,\]
	where $T$ is a torus and $A$ is an abelian variety. We denote by $\cgv$ the category of commutative group varieties over $\qb$ and by $\CT$, $\CA$ and ${\CS \CA}$ the full subcategories of $\cgv$ of tori, abelian varieties and semiabelian varieties. Moreover, we denote by $\CT_\BQ$, $\CA_\BQ$ and ${\CS \CA}_\BQ$ the categories of tori, abelian varieties and semiabelian varieties \textit{up to isogeny}, i.e., with morphisms tensored with $\BQ$. Recall that, for all semiabelian varieties $G$ and $G'$, we also have
	\[\Ext{\sav}{G}{G'} \cong \Ext{\cgv}{G}{G'} \otimes_\BZ \BQ.\]
	The category $\sav$ is $\BQ$-linear and abelian, and its full abelian subcategories $\CT_\BQ$ and $\CA_\BQ$ are even semisimple. For more on semiabelian varieties see \cite[§4]{HW}, \cite{BrionIsogeny}, \cite[§1.4]{NesaPhD}.
\end{conv}

\begin{defn}[{\cite[D\'efinition (10.1.2)]{Del3}}] \label{defn:Omot}
	A \textit{$1$-motive} $M$ over $\qb$ consists of:
	\begin{itemize}
		\item a finite free $\BZ$-module $X$, an abelian variety $A$ and an algebraic torus $T$;
		\item a semiabelian variety $G$ which is an extension of $A$ by $T$;
		\item a group homomorphism $u:X \to G(\qb)$.
	\end{itemize}
	The $\BQ$-linear abelian category $\omot$ of $1$-motives has as objects $1$-motives and as arrows morphisms of complexes, tensored with $\BQ$.
\end{defn}

\begin{conv}
	Other sources (see for instance \cite[§8]{HW}) distinguish the categories of $1$-motives and iso-$1$-motives, i.e., $1$-motives up to isogeny. We will only consider the second category, and call its objects simply \textit{$1$-motives}.
\end{conv}

\begin{conv}
	We use the notation $M=[X \to G]$, or sometimes the more explicit $M=(X,A,T,G,u)$. If $X$ or $G$ is trivial, we denote $M$ as $[G]$ or $[X]$, respectively.
\end{conv}

\begin{para} \label{para:Periods}
	Since we are going to talk of dimensions of spaces of periods of $1$-motives, we should at least sketch what periods are. The idea one should have in mind is that a period is a complex number obtained integrating some algebraic differential form on some algebraic cycle, even though with motives it is not always straightforward to see on what space the integration is taking place. For the study of periods of $1$-motives, the relevant \textit{realization functors} are the singular realization and the de Rham realization, which are linear, faithful, exact functors $V_\mathrm{sing}(\cdot): \omot \to \vects{\BQ}$ and $V_\mathrm{dR}(\cdot): \omot \to \vects{\qb}$ (the targets are the categories of finite dimensional vector spaces over $\BQ$ and $\qb$, respectively). These two realizations become isomorphic after base change to $\BC$, so for every $1$-motive $M$ there is a $(\BQ,\qb)$-bilinear and functorial \textit{integration pairing}
	\[ \int: V_\mathrm{sing}(M) \times V_\mathrm{dR}(M)^* \to \BC.\]
	The \textit{periods of $M$} are the elements of the finite dimensional $\qb$-subspace $\CP(M) \seq \BC$ generated by the image of this pairing. By functoriality, the periods of all the motives in the category $\abgen{M}$ generated by $M$ (see Convention \ref{conv:Cats}) are contained in $\CP(M)$. The $\qb$-dimension of $\CP(M)$ is denoted $\delta(M)$.
	
	We do not need to go into more details for the purpose of this article, so we refer the interested reader to the original reference \cite[§10.1]{Del3} and to the more detailed explanations in \cite[§8.1, §9.1]{HW}. We limit ourselves to presenting a few examples of $1$-motives and of their periods.
\end{para}
	
\begin{ex} \label{ex:Simple}
	The simplest non-trivial $1$-motives are those with exactly one non-trivial component, i.e., those of the form $[X]$; $[T]$; or $[A]$. Their periods are generated over $\qb$, respectively, by $1$; $2\pi i$; or the (classical) periods of $A$. See \cite[Example 9.5; §10.1; §10.4]{HW} for more details.
\end{ex}

\begin{ex}\label{ex:Baker}
	The next easiest examples are \textit{Baker motives}, i.e., $1$-motives for which only $A$ is $0$, which are thus of the form $M=[X \to T]$. Such a $1$-motive has $[T]$ as a subobject and $[X]$ as a quotient. We can rewrite $M$ without loss of generality as
	\[M=[\BZ^r \xrightarrow{(\alpha_{jk})} \BG_m^s],\]
	where $\alpha_{jk}$ ($j \in \{1,\dots,s\}$, $k \in \{1,\dots,r\}$) denotes the algebraic number which is the image of $(0,\dots,0,1,0,\dots,0) \in \BZ^r$ (only the $k$-th entry is $1$) in the $j$-th component of $\BG_m^s(\qb) = (\qb^\times)^s$. We can thus think of the map $X \to T$ as the matrix $(\alpha_{jk})$ in $\qb^{s \times r}$. Then, the periods of $M$
	are  generated by $1,2\pi i,\{\log(\alpha_{jk})\}_{jk}$ over $\qb$, where ``$\log$'' can be any branch of the complex logarithm (since they differ by integer multiples of $2\pi i$). For the computations, see \cite[§10.2]{HW}.
\end{ex}

We define ranks of $1$-motives in Definition \ref{defn:Rank}, in the next section. But now that we have sketched the setting, to provide more context we formulate \cite[Theorem 1.4]{HW}, where ranks of $1$-motives originally appear. For a $1$-motive $M$, recall that we denote by $\delta(M)$ the dimension of its linear space of periods $\CP(M)$.

\begin{thm}[{\cite[Theorem 1.4]{HW}}]
	\label{thm:Formula}
	Let $M=(X,A,T,G,u)$ be a $1$-motive. Let $A \cong B_1^{n_1} \times \cdots \times B_m^{n_m}$ be the decomposition of $A$ into simple factors in the category of abelian varieties up to isogeny, and for each $B_i$ denote by $g(B_i)$ the dimension of $B_i$ and by $e(B_i)$ the dimension of $\End{}{B_i}_\BQ$ over $\BQ$. The linear space $\CP(M)$ of periods of $M$ has dimension
	\[
	\delta(M) = \delta_{\textnormal{Ta}}(M) + \delta_2(M) + \delta_\textnormal{alg}(M) + \delta_3(M) + \delta_\textnormal{inc2}(M) + \delta_\textnormal{inc3}(M),
	\]
	where
	\begin{itemize}
		\item $\delta_{\textnormal{Ta}}(M) = \delta([T])$ is $0$ if $T$ vanishes, and $1$ otherwise;
		\item $\delta_\textnormal{alg}(M) = \delta([X])$ is $0$ if $X$ vanishes, and $1$ otherwise;
		\item $\delta_2(M) = \delta([A]) = \sum_{B_i} \tfrac{4g(B_i)^2}{e(B_i)}$;
		\item $\delta_3(M) = \sum_{B_i}2g(B_i)\rk{B_i}{X}{M}$ (see Definition \ref{defn:Rank});
		\item $\delta_\textnormal{inc2}(M) = \sum_{B_i}2g(B_i)\rk{B_i}{T}{M}$ (see Definition \ref{defn:Rank});
		\item $\delta_\textnormal{inc3}(M)$ is as described in \cite[§17]{HW}.
	\end{itemize}
	All sums are over all simple factors $B_i$ of $A$, up to isogeny, without multiplicities.
\end{thm}

\begin{para}
	The goal of this article is to shed light on $\delta_\textnormal{inc2}$ and $\delta_3$, or more precisely, on the ranks that appear in their formulas. We do not say much about $\delta_\textnormal{inc3}(M)$, which in general is more complicated (except if $X$, $T$ or $A$ are trivial; see also Corollary \ref{cor:Baker}) and is computed in detail with \textit{ad hoc} techniques in \cite[§17]{HW}. Subsequent work by Huber and Kalck (see \cite{HK}) makes also the description of $\delta_\textnormal{inc3}(M)$ more categorical.
\end{para}

To conclude this recapitulation, let us now recall a few facts and definitions about extensions and dualities.

\begin{para} \label{para:ExtChar}
	Removing the zero section from a line bundle in $\mathrm{Pic}^0(A) \cong A^*(\qb)$ (where $A^*$ is the abelian variety dual to $A$, see \cite[§8]{Mumford}) produces a $\BG_m$-bundle on $A$ which is a semiabelian variety; this induces an isomorphism
	\[A^*(\qb) \isom \Ext{\cgv}{A}{\BG_m}.\]
	More generally, given a torus $T$ and an abelian variety $A$, we have an isomorphism of $\BZ$-modules
	\[ \Ext{\cgv}{A}{T} \isom \Hom{\ab}{\chi(T)}{A^*(\qb)},\]
	where on the right-hand side we have morphisms of abelian groups and $\chi(T)$ denotes the character group of $T$. The isomorphism sends an extension $G$ to the map that assigns to a character $\phi:T \to \BG_m$ the extension
	\[ \phi G \in \Ext{\cgv}{A}{\BG_m} \cong A^*(\qb)\]
	(see \cite[Corollary 4.11]{HW} and the notation of \ref{para:Ext}).
\end{para}

\begin{defn}[{\cite[(10.2.11)]{Del3}}] \label{defn:CartierDual}
	Let $M=(X,A,T,G,u)$ be a $1$-motive. The \textit{Cartier dual} $M^*=(X^*,A^*,T^*,G^*,u^*)$ of $M$ is defined as follows:
	\begin{itemize}
		\item $X^*:= \Hom{\cgv}{T}{\BG_m}$ is the group of characters of $T$;
		\item $A^* \cong \Ext{\cgv}{A}{\BG_m}$ is the dual abelian variety of $A$;
		\item $T^*$ is the torus whose group of characters is $X$;
		\item $G^*:=\Ext{\omot}{[X \to A]}{\BG_m}$ (see \cite[(10.2.11) b)-c)]{Del3});
		\item $u^*:X^* \to G^*(\qb)$ maps $x:T \to \BG_m$ to the extension of $[X \to A]$ by $\BG_m$ obtained by push-out of $M$ along $x$, where $M$ is seen as an element of $\Ext{\omot}{[X \to A]}{T}$ (similarly to \ref{para:ExtChar}).
	\end{itemize}
\end{defn}

\begin{conv}
	In this article, the symbol $^*$ is only meant to recognize the constituents of $M^*$, and not to denote duality on all of them individually (except for $A^*$).
\end{conv}

\begin{rem} \label{rem:GeneralExtensionG}
	The map $X^* \to G^*(\qb) \to A^*(\qb) = \Ext{\cgv}{A}{\BG_m}$ parameterizes precisely the extension $G$ of $A$ by $T$, as in \ref{para:ExtChar}. In particular, $X^* \to A^*(\qb)$ is torsion if and only if $G$ is a split extension, in the sense that there exists a section of $G \to A$ up to isogeny. For a more detailed discussion of the relations between a $1$-motive and its dual, see \cite[pp.3-4]{Bertrand} and \cite[§2.1]{NesaPhD}.
\end{rem}

\section{Results for ranks of $1$-motives}
\label{s:results}

In this section we apply Proposition \ref{prop:ExtGeneral} to deduce Propositions \ref{prop:RankExtOmotBZ}, \ref{prop:RankExtOmot} and \ref{prop:RankExtOmotX}, which give a simpler description of the ranks of $1$-motives (see Definition \ref{defn:Rank}) as dimensions of Ext$^1$ vector spaces. These ranks are crucial elements in the computation of the general dimension formula for the linear space of periods of a $1$-motive in \cite[Theorem 1.4]{HW}, which we have reproduced in Theorem \ref{thm:Formula} for convenience.

\begin{conv}\label{conv:Isogeny}
	In the rest of this article, semiabelian varieties are always considered as elements of $\sav$ (i.e., up to isogeny, see Convention \ref{conv:CatsSAV}). In particular, if $A$ is an abelian variety, by $\End{}{A}$ we mean $\End{\sav}{A}$, which is $\End{\CS \CA}{A} \otimes \BQ$ and is a skew-field if $A$ is simple. Moreover, if we write that an extension of semiabelian varieties is split, we mean that there exists a section up to isogeny.
\end{conv}

Now we have introduced all the ingredients needed to define the ranks of $1$-motives.

\begin{defn}[{\cite[Notation 15.2]{HW}}] \label{defn:Rank}
	Let $M =(X,A,T,G,u)$ be a $1$-motive.
	\begin{enumerate}
		\item If $A \neq 0$, let $B$ be a simple component of $A$; for $\phi \in \Hom{}{A}{B}$, we denote by $\phi(X)$ the image of $X_\BQ$ in $B(\qb)_\BQ$ under the composition of the map $X_\BQ \to A(\qb)_\BQ$ with $\phi$. The $\BQ$-subspace
		\[\sum_{\phi \in \Hom{}{A}{B}} \phi(X) \seq B(\qb)_\BQ\]
		is a finite dimensional left $\End{}{B}$-vector space. The \textit{$X$-rank of $M$ with respect to $B$} is its dimension:
		\[ \rk{B}{X}{M}:=\dim_{\End{}{B}}{\left(\sum_{\phi \in \Hom{}{A}{B}} \phi(X) \right)}.\]
		\item If $A \neq 0$, let $B$ be a simple component of $A$ and consider the right action of $\End{}{B}$ on the dual abelian variety $B^*$ given by $(b, f) \mapsto f^*(b)$ (which corresponds to the usual left action of $\End{}{B^*}$ on $B^*$ under the natural isomorphism sending $f\in \End{}{B}$ to $f^* \in \End{}{B^*}$). For $\psi \in \Hom{}{B}{A}$, we denote by $\psi^*(X^*)$ the image of $X^*_\BQ$ ($X^*$ is the group of characters of $T$, see Definition \ref{defn:CartierDual}) in $B^*(\qb)_\BQ$ under the composition of the morphism $X^*_\BQ \to A^*(\qb)_\BQ$ with $\psi^*$. The $\BQ$-subspace
		\[\sum_{\psi \in \Hom{}{B}{A}} \psi^*(X^*) \seq B^*(\qb)_\BQ\]
		is a finite dimensional right $\End{}{B}$-vector space. The \textit{$T$-rank of $M$ with respect to $B$} is its dimension:
		\[ \rk{B}{T}{M}:=\dim_{\End{}{B}}{\left(\sum_{\psi \in \Hom{}{B}{A}} \psi^*(X^*) \right)}.\]
		\item If $A=0$, for $\phi \in X^*$ we denote by $\phi(X)$ the image of $X$ in $\BG_m(\qb)$ under the composition of the morphism $X \to T(\qb)$ with $\phi$. The \textit{$X$-rank $\rk{\BG_m}{X}{M}$ of $M$ with respect to $\BG_m$} is the rank of the subgroup
		\[ \left(\sum_{\phi \in X^*} \phi(X) \right) \seq \BG_m(\qb).\]
	\end{enumerate}
\end{defn}

\begin{rem}
	Even though $\rk{\BG_m}{X}{M}$ does not appear explicitly in the formula of Theorem \ref{thm:Formula}, for Baker motives it equals $\delta_{\mathrm{inc3}}(M)$ (see Corollary \ref{cor:Baker}).
\end{rem}

\begin{lem} \label{lem:RankDual}
	Let $M=(X,A,T,G,u)$ be a $1$-motive. If $A\neq0$ and $B$ is a simple component of $A$, then
	\[ \rk{B}{X}{M} = \rk{B^*}{T^*}{M^*} \hand \rk{B}{T}{M} = \rk{B^*}{X^*}{M^*}.\]
\end{lem}

\begin{proof}
	Transform the right-hand side of the equalities into the left-hand side using that $\Hom{}{A}{B}$ is naturally isomorphic to $\Hom{}{B^*}{A^*}$ and that double duality of $1$-motives is naturally isomorphic to the identity.
\end{proof}

\begin{para}
	From Definition \ref{defn:Rank} we see directly that, if $A=0$, then $\rk{\BG_m}{X}{M}=0$ if and only if the map $X \to T(\qb)$ is torsion, i.e., if and only if $[X \to T]$ is a split extension of $[X]$ by $[T]$. And we also see that, if $A \neq 0$, then $\rk{B}{T}{M}=0$ for all simple components $B$ of $A$ if and only if $X^* \to A^*(\qb)$ is torsion, i.e., if and only if $G$ is a split extension of $A$ by $T$ (see Remark \ref{rem:GeneralExtensionG}). These are the first hints that there is some relation between the ranks of Definition \ref{defn:Rank} and some Ext$^1$ vector spaces. In the rest of this section, our goal is to unravel completely this relation.
\end{para}

\begin{para} \label{para:MakeThingsClear}
	Let $M=(X,A,T,G,u)$ be a $1$-motive with $A=0$, $T \neq 0$, $X \neq 0$. Without loss of generality, let $X=\BZ^r$ and $T=\BG_m^s$ and consider the maps
	\[ u_{ij}: \BZ \xrightarrow{\iota_i} X \xrightarrow{u} T \xrightarrow{\rho_j} \BG_m,\]
	where $\iota_i$ denotes the inclusion into the $i$-th coordinate and $\rho_j$ denotes the projection onto the $j$-th coordinate. The following equality of subgroups of $\BG_m(\qb)$ holds:
	\[ \sum_{\phi \in X^*} \phi(X) = \sum_{i,j} \img{{u_{ij}}}.\]
	Now consider the composite morphism of $\BQ$-vector spaces
	\[ \Ext{\la M \ra}{[\BZ]}{[\BG_m]} \inj  \Ext{\omot}{[\BZ]}{[\BG_m]} \isom \BG_m(\qb)_\BQ.\]
	Going through the definitions (see \ref{para:Ext}), we see that, for all $i$ and $j$, the point $u_{ij}(1)$ in $\BG(\qb)_\BQ$ (on the right-hand side) is the image of the extension
	\[ [\BZ\xrightarrow{u_{ij}} \BG_m] = \rho_j [X \xrightarrow{u} T] \iota_i = \rho_j M \iota_i \quad \in \Ext{\abgen{M}}{[\BZ]}{[\BG_m]}.\]
	Thus, more generally, we conclude that the $\BQ$-vector space
	\[ \left(\sum_{\phi \in X^*} \phi(X) \right)_\BQ \]
	is canonically isomorphic to the subspace
	\[R_{\BG_m} := \left\{\sum_{i=1}^n \phi_i M \psi_i \ \middle\vert \ n \in \BN, \phi_i \in \Hom{\omot}{[T]}{[\BG_m]}, \psi_i \in \Hom{\omot}{[\BZ]}{[X]}\right\}\]
	of $\Ext{\abgen{M}}{[\BZ]}{[\BG_m]}$. Now we show that $R_{\BG_m}$ is actually the whole $\Ext{\abgen{M}}{[\BZ]}{[\BG_m]}$.
\end{para}

\begin{prop} \label{prop:RankExtOmotBZ}
	Let $M=(X,A,T,G,u)$ be a $1$-motive with $A = 0$, $T \neq 0$ and $X \neq 0$. Then
	\[ \rk{\BG_m}{X}{M} = \dimq {\Ext{\la M \ra}{[\BZ]}{[\BG_m]}}.\]
\end{prop}

\begin{proof}
	We claim that every element of $\Ext{\la M \ra}{[\BZ]}{[\BG_m]}$ is of the form
	\[\sum_{i=1}^n \phi_i M \psi_i, \quad \text{for }n \in \BN, \phi_i \in \Hom{}{[T]}{[\BG_m]}, \psi_i \in \Hom{}{[\BZ]}{[X]}.\]
	This follows from Proposition \ref{prop:ExtGeneral}, since in the abelian category $\omot$ there is no non-zero morphism from the semisimple abelian subcategory of algebraic tori to the semisimple abelian subcategory of finite free $\BZ$-modules. So the vector space $R_{\BG_m}$ of \ref{para:MakeThingsClear} is the whole $\Ext{\la M \ra}{[\BZ]}{[\BG_m]}$, and (by \ref{para:MakeThingsClear} and by Definition \ref{defn:Rank}) $\rk{\BG_m}{X}{M}$ is its dimension.
\end{proof}

\begin{rem}
	If $T=0$ or $X=0$, then $\Ext{\la M \ra}{[\BZ]}{[\BG_m]}$ is not defined, because one of $[\BZ]$ or $[\BG_m]$ is not in $\abgen{M}$.
\end{rem}

\begin{para} \label{para:MakeThingsClearB}
	Let $M=(X,A,T,G,u)$ be a $1$-motive with $A \neq 0$ and $T \neq 0$, and let $B$ be a simple component of $A$. The following reasoning is similar to that of \ref{para:MakeThingsClear}, but it uses the identification of $\End{}{B}$-vector spaces
	\[B^*(\qb)_\BQ \cong \Ext{\sav}{B}{\BG_m}\]
	(see \ref{para:ExtChar}). Now recall that $X^*_\BQ = \Hom{}{T}{\BG_m}$ and that the map $X^* \to A^*(\qb)$ corresponds to the choice of the semiabelian variety $G$ of $M$ (see Remark \ref{rem:GeneralExtensionG}). Going through the definitions, as in \ref{para:MakeThingsClear} we see that under the above identification the $\End{}{B}$-subspace
	\[ \sum_{\psi \in \Hom{}{B}{A}} \psi^*(X^*) \seq B^*(\qb)_\BQ\]
	corresponds to the $\End{}{B}$-subspace
	\[ R_B:= \left\{\sum_{i=1}^n \phi_i G \psi_i \ \middle\vert \ n \in \BN, \phi_i \in \Hom{}{T}{\BG_m}, \psi_i \in \Hom{}{B}{A}\right\}\]
	of $\Ext{\abgen{G}}{B}{\BG_m}$. Also in this case, we show that $R_B$ is actually the whole $\Ext{\abgen{G}}{B}{\BG_m}$.
\end{para}

\begin{prop} \label{prop:RankExtOmot}
	Let $M=(X,A,T,G,u)$ be a $1$-motive with $A \neq 0$ and $T \neq 0$, and let $B$ be a simple component of $A$. Then
	\[ \rk{B}{T}{M} = \dim_{\End{}{B}} {\left(\Ext{\la G \ra}{B}{\BG_m}\right)} = \dim_{\End{}{B}} \left(\Ext{\la M \ra}{[B]}{[\BG_m]} \right).\]
\end{prop}

\begin{proof}
	Like for Proposition \ref{prop:RankExtOmotBZ}, but with categories $\sav$, $\CT_\BQ$, $\CA_\BQ$. The last equality follows from the fact that the two Ext$^1$ vector spaces are identified under the isomorphism of $\End{}{B}$-vector spaces $\Ext{\sav}{B}{\BG_m} \cong \Ext{\omot}{[B]}{[\BG_m]}$.
\end{proof}

There is also a ``dual version'' of Proposition \ref{prop:RankExtOmot}:

\begin{prop} \label{prop:RankExtOmotX}
	Let $M=(X,A,T,G,u)$ be a $1$-motive with $A \neq 0$ and $X \neq 0$, and let $B$ be a simple component of $A$. Then
	\[ \rk{B}{X}{M} = \dim_{\End{}{B}} {\left(\Ext{\la M \ra}{[\BZ]}{[B]}\right)}.\]
\end{prop}

\begin{proof}
	From Lemma \ref{lem:RankDual} and from Proposition \ref{prop:RankExtOmot} applied to $M^*$ we deduce that
	\[ \rk{B}{X}{M} = \dim_{\End{}{B}} {\left(\Ext{\la G^* \ra}{B^*}{\BG_m}\right)}\]
	(since $\End{}{B}$ and $\End{}{B^*}$ are canonically isomorphic through duality). It thus suffices to show that
	\[ \Ext{\la G^* \ra}{B^*}{\BG_m} \cong \Ext{\la M \ra}{[\BZ]}{[B]}.\]
	The extension $G^*$ of $A^*$ by $T^*$ corresponds to the extension $[X \to A]$ of $[X]$ by $[A]$ (see Remark \ref{rem:GeneralExtensionG}). For every $\phi \in \Hom{}{T^*}{\BG_m}$ and $\psi \in \Hom{}{B^*}{A^*}$ we obtain an extension $\phi G^* \psi$ in the space on the left-hand side which corresponds to an extension $\psi^* [X \to A] \phi^*$ in the space on the right-hand side. Applying Proposition \ref{prop:ExtGeneral} to both sides, we see that such extensions generate the two spaces.
\end{proof}

\begin{rem}
	In Proposition \ref{prop:RankExtOmotBZ} we compute a $\BQ$-dimension, whereas in Propositions \ref{prop:RankExtOmot} and \ref{prop:RankExtOmotX} we compute $\End{}{B}$-dimensions. This apparent discrepancy is simply due to the fact that $\End{\omot}{[\BG_m]} = \BQ$. In fact, since $\End{\omot}{[\BZ]}=\BQ$ as well, in all three propositions what we compute is actually the dimension of a free $\End{\CB}{A'} \otimes \End{\CB}{C'}^\mathrm{op}$-module of the form $\Ext{\abgen{B}}{C'}{A'}$. This point of view is developed further in \cite{HK}.
\end{rem}

\section{Corollaries and examples}

\begin{cor} \label{cor:SimpleRank1}
	Let $M=(X,A,T,G,u)$ be a $1$-motive.
	\begin{itemize}
		\item If $A$ and $T$ are simple and $G$ is non-split, then $\rk{A}{T}{M}=1$.
		\item If $A$ and $X$ are simple and $G$ is non-split, then $\rk{A}{X}{M}=1$.
		\item If $A=0$, $T$ and $X$ are simple and $X \to T(\qb)$ is non-torsion, then $\rk{\BG_m}{X}{M}=1$.
	\end{itemize}
\end{cor}

\begin{proof}
	If $T$ is simple it is isomorphic to $\BG_m$, so $\Hom{}{T}{\BG_m} = \BQ$, which implies that the space $R_A$ in \ref{para:MakeThingsClearB} is generated by the element $G$ (which is non-trivial by assumption) as a right $\End{}{A}$-vector space. Therefore, it has dimension $1$. The other cases are similar.
\end{proof}

\begin{cor}
	\label{cor:Baker}
	Let $M=[X \to T]$ be a Baker motive, i.e., a $1$-motive with $A=0$, $T \neq 0$ and $X \neq 0$. The linear space of periods of $M$ has dimension
	\[ \delta(M) = \delta([T]) + \delta([X]) + \rk{\BG_m}{X}{M} = 2+ \dimq {\Ext{\la M \ra}{[\BZ]}{[\BG_m]}}.\]
\end{cor}

\begin{proof}
	Combine \cite[Proposition 15.10]{HW} and Proposition \ref{prop:RankExtOmotBZ}.
\end{proof}

\begin{rem}
	We have mentioned in Example \ref{ex:Baker} that the periods of a Baker motive
	\[M=[\BZ^r \xrightarrow{(\alpha_{jk})} \BG_m^s],\]
	are spanned over $\qb$ by $1,2\pi i$ and $\{\log(\alpha_{jk})\}_{jk}$. Since we have also seen in Example \ref{ex:Simple} that the periods of $[X]$ and $[T]$ are spanned, respectively, by $1$ and $2\pi i$, and since both $2\pi i$ and $\log(\alpha_{jk})$ are transcendental (if $\alpha_{jk}\neq 1$), we deduce from Corollary \ref{cor:Baker} that
	\[ \rk{\BG_m}{X}{M} = \dimqb{\qb\la2\pi i, \log(\alpha_{11}),\dots,\log(\alpha_{sr})\ra}-1.\]
	Note that we need to keep (and then subtract) the contribution of $2\pi i$ for the eventuality that it is linearly dependent from the $\log(\alpha_{jk})$'s. Going back to Definition \ref{defn:Rank}, where $\rk{\BG_m}{X}{M}$ was introduced, we see that
	\[ \rk{\BG_m}{X}{M} = \mathrm{rk}_\BZ\la \alpha_{11}, \dots,\alpha_{sr} \ra,\]
	where $\la \alpha_{11}, \dots,\alpha_{sr} \ra$ denotes the multiplicative subgroup of $\BG_m(\qb)$ spanned by the $\alpha_{jk}$'s. Putting these two equalities together we recover Baker's theorem about linear forms in logarithms of algebraic numbers, which (as reformulated in \cite[Theorem 10.5]{HW}) states precisely that, given algebraic numbers $\alpha_1,\dots,\alpha_n \in \qb^\times$, we have
	\[ \dimqb{\qb\la2\pi i, \log(\alpha_1),\dots,\log(\alpha_n)\ra}-1= \mathrm{rk}_\BZ\la \alpha_1, \dots,\alpha_n \ra.\]
	The original reference is \cite[Corollary 1]{Bak}.
\end{rem}

\begin{para}
	In this article we have been concerned with Ext$^1$ groups of $1$-motives. For the sake of completeness, let us conclude with a few remarks about higher Ext groups of $1$-motives. The category of $1$-motives has cohomological dimension $1$ (see \cite[Proposition 3.2.4]{Orgogozo}), so Ext$^n_\omot$ groups vanish for $n > 1$. However, if we restrict to the abelian subcategory $\abgen{M}$ generated by a $1$-motive $M$, this might not be the case. In fact, Huber and Kalck show in \cite{HK} that higher Ext groups play a crucial role in simplifying the summand $\delta_\mathrm{inc3}(M)$ in the formula of \cite[Theorem 1.4]{HW}. For this reason, we want to conclude this article with an example of a non-vanishing Ext$^2_\abgen{M}$ group (thanks to Annette Huber for suggesting this example). To explain it, we first need a general lemma. Recall from Convention \ref{conv:Isogeny} that we are considering semiabelian varieties up to isogeny.
\end{para}

\begin{lem} \label{lem:HomAG}
	Let $A\neq 0$ be an abelian variety and let $G$ be an extension of $A$ by a torus $T$. If $G$ is split, then $\Hom{}{A}{G}\neq 0$. If $A$ is simple, the two conditions are equivalent.
\end{lem}

\begin{proof}
	Denote by $\pi:G \to A$ the structure map. If $G$ is split, by definition there is a section $\psi:A \to G$ of $\pi$, so $\Hom{}{A}{G} \neq 0$. Conversely, if $A$ is simple and there is a non-trivial map $\psi:A \to G$, then $\alpha:= \pi \circ \psi$ cannot be zero because otherwise $\psi$ would factor through $T$, but there are no non-zero maps from an abelian variety to a torus. So $\alpha \neq 0$; but since $A$ is simple, $\End{}{A}$ is a division ring, so $\alpha$ is invertible. In particular, $\psi \circ \alpha^{-1} :A \to G$ is a section of $\pi$, so $G$ is a split extension.
\end{proof}

\begin{prop}
	Let $T$ be a non-trivial torus, $A$ a simple abelian variety, and $G$ a non-split extension of $A$ by $T$. Let $x$ be a non-torsion point of $A(\qb)$ and let $M$ be the $1$-motive
	\[ M:= [\BZ \xrightarrow{x} A] \oplus [G] = [\BZ \xrightarrow{(0,x)} G \oplus A].\]
	Then
	\[\Extn{\abgen{M}}{2}{[\BZ]}{[T]} \neq 0.\]
\end{prop}

\begin{proof}
	Assume, by contradiction, that $\Extn{\abgen{M}}{2}{[\BZ]}{[T]}$ vanishes. From the exact sequence
	\[ \cdots \to \Ext{\abgen{M}}{[\BZ]}{[G]} \to \Ext{\abgen{M}}{[\BZ]}{[A]} \to \Extn{\abgen{M}}{2}{[\BZ]}{[T]} \to \cdots\]
	we see that, under this assumption, the map
	\[\Ext{\abgen{M}}{[\BZ]}{[G]} \to \Ext{\abgen{M}}{[\BZ]}{[A]}\]
	is surjective. Notice that $\Ext{\abgen{M}}{[\BZ]}{[A]}$ contains $[\BZ\xrightarrow{x} A]$, which is non-split by definition of $x$. So, to reach a contradiction and conclude the proof, it is sufficient to show that every $1$-motive of the form $[\BZ \to G]$ in $\abgen{M}$ splits as $[\BZ]\oplus [G]$, i.e., that $\Ext{\abgen{M}}{[\BZ]}{[G]}=0$. Assume, again by contradiction, that there is a non-split $1$-motive $[ \BZ \xrightarrow{z} G]$ in $\abgen{M}$. In particular, the $1$-motive $[ \BZ \xrightarrow{z} G]$ is a subquotient of a power of $M$, so there are $r,n \in \BN$, a semiabelian variety $G'$, and a pair of morphisms of $1$-motives of the form
	\[ [ \BZ \xrightarrow{z} G] \twoheadleftarrow [\BZ^r \to G'] \inj M^n.\]
	Written out more explicitly, this corresponds to a commutative diagram of the form
	\[\begin{tikzcd}
		\BZ \arrow[d, hook, "z"] & \BZ^r \arrow[r, hook] \arrow[d] \arrow[l, two heads] & \BZ^n \arrow[d, hook, "{(0,x^n)}"] \\
		G & G' \arrow[r, hook] \arrow[l, two heads] & G^n \oplus A^n.
	\end{tikzcd}\]
	By surjectivity and commutativity of the maps on the left, we can pull back the non-torsion point $z \in G$ to a non-torsion point $z' \in G'$ in the image of $\BZ^r$. By injectivity and commutativity of the maps on the right, $z'$ maps to a non-torsion point $z''$ of $\{0\} \oplus A^n$. In particular, $G' \cong \img{G'}$ intersects $\{0\} \oplus A^n$ non-trivially in an abelian subvariety $B$ containing $z''$. Since the category of abelian varieties up to isogeny is semisimple, there is a retraction of the inclusion $B \inj A^n$. So there is a composite map
	\[ A^n \surj B \inj G' \surj G,\]
	which is non-trivial because it maps $z'' \mapsto z$. This implies that $\Hom{}{A}{G}\neq0$, which contradicts Lemma \ref{lem:HomAG}, since $G$ is non-split and $A$ is simple. So we have reached a contradiction and we have proved that
	\[ \Extn{\abgen{M}}{2}{[\BZ]}{[T]} \neq 0.\]
\end{proof}

\bibliographystyle{amsalpha}

\end{document}